\documentclass[11pt]{amsart}

\usepackage{fullpage,graphicx,amsfonts,amssymb,amsmath,amsthm}
\usepackage[all]{xy}
\usepackage[left=.8in,top=.8in,bottom=.8in,right=.8in,letterpaper]{geometry} 
\usepackage{mathtools}
\usepackage{graphicx}
\graphicspath{ {images/} }
\usepackage{enumerate}
\usepackage{hyperref} 
\usepackage{setspace}
\usepackage{amssymb} 
\usepackage{ esint }

\allowdisplaybreaks
\theoremstyle{plain} %--default
\newtheorem{theorem}    {Theorem}

\newtheorem{lemma}      [theorem]{Lemma}
\newtheorem{corollary}  [theorem]{Corollary}
\newtheorem{proposition}[theorem]{Proposition}

\theoremstyle{definition}

\theoremstyle{remark}

\newcommand\norm[1]{\left\lVert#1\right\rVert}
\usepackage{url}

\begin{document} 

\title{The Isomorphic Section-Projection Problem for Convex Bodies}

\author{Johannes Hosle \\ 07/12/2026}

\address{Department of Mathematics, Massachusetts Institute of Technology}
\email{jhosle@mit.edu}
\maketitle 

\begin{abstract}
Let $K, L$ be convex bodies in $\mathbb{R}^n$ with $K$ centered. Assume that $|K \cap \theta^{\perp}| \le |L|\theta^{\perp}|$ for all $\theta \in S^{n-1}$. We prove that $|K| \le c\sqrt{n}|L|$, which is sharp up to the choice of the absolute constant. The result gives the sharp isomorphic order in a mixed section-projection comparison problem, complementing the isomorphic Busemann-Petty and Shephard problems. It also removes the John's position assumption from an earlier result of the author, up to an absolute constant factor.
\end{abstract}

\section{Introduction}
The subject of comparing volumes of convex sets via comparisons of lower dimensional measurements is a central theme in convex geometry. The Busemann-Petty problem, posed in \cite{busemannpetty}, is the following question:

\textit{If $K$ and $L$ are origin-symmetric convex bodies in $\mathbb{R}^n$ such that \begin{align*}
    |K \cap \theta^{\perp}| \le |L \cap \theta^{\perp}|
\end{align*} for all $\theta \in S^{n-1}$, does it follow that $|K| \le |L|$?} 

The answer is affirmative in dimensions $n\le 4$ and negative in dimensions 5 and higher. The disproof for $n\ge 5$ was carried out by Papadimitrakis \cite{papa}, Gardner \cite{Gardner1}, and Zhang \cite{Zhang1}, while the affirmative answers for $n = 3, 4$ were given by Gardner \cite{Gardner2} and Zhang \cite{Zhang2} respectively. Note that the case $n = 2$ is immediate, since the assumptions imply $K \subset L$. A uniform solution for all dimensions using the method of Fourier transforms on the sphere was accomplished by Gardner, Koldobsky, and Schlumprecht \cite{GKS}.

The Shephard problem, posed in \cite{shepard}, asks the analogous question about projections onto hyperplanes, rather than hyperplane sections:

\textit{If $K$ and $L$ are origin-symmetric convex bodies in $\mathbb{R}^n$ such that \begin{align*}
    |K|{\theta^{\perp}}| \le |L|{\theta^{\perp}}|
\end{align*} for all $\theta \in S^{n-1}$, does it follow that $|K| \le |L|$?}

As in the case of the Busemann-Petty problem, the answer to this question is negative in general. The claim only holds in dimensions $n = 2$, as was demonstrated by Schneider \cite{schneider} and Petty \cite{petty}.

A variant of the original Busemann-Petty and Shephard problems asks if, under the same hypotheses, one can conclude $|K| \le C(n) |L|$ for a dimensional constant $C(n)$, and it asks for the optimal order of magnitude of $C(n)$. These problems are known as the isomorphic Busemann-Petty problem and isomorphic Shephard problem respectively. For the first question, Milman and Pajor \cite{milmanpajor} gave the estimate $|K| \le cL_K |L|,$ where $L_K$ is the isotropic constant of $K$ and $c>0$ is an absolute constant. The question of whether $L_K$ is uniformly bounded for convex bodies $K$ in any dimension was known as the \textit{Bourgain slicing problem}, which has now been answered affirmatively by Klartag and Lehec \cite{KlartagLehec2025SlicingGuan}. As an immediate corollary, it follows that $|K| \le c |L|$ under the assumptions of the Busemann-Petty problem. For the isomorphic Shephard problem, Ball \cite{ballshadows} gave the bound $|K| \le c\sqrt{n} |L|.$ Moreover, by considering random polytopes, Ball showed that his bound was sharp, that is, for all $n$, there exist $K, L$ origin-symmetric convex bodies in $\mathbb{R}^n$ with $|K|\theta^{\perp}| \le |L|\theta^{\perp}|$ for all $\theta \in S^{n-1}$ such that $|K| \ge c\sqrt{n}|L|$, for some absolute constant $c > 0$.

We will now discuss two natural variants of the (isomorphic) Busemann-Petty and Shephard problems. Both use mixed section and projection hypotheses. Giannopoulos and Koldobsky \cite{giankold} proved that if $K, L$ are origin-symmetric convex bodies in $\mathbb{R}^n$ where now $|K|\theta^{\perp}| \le |L \cap \theta^{\perp}|$ for all $\theta \in S^{n-1}$, then $|K| \le |L|$. Hosle \cite{Hosle} considered the reverse question. He showed that if $K, L$ are origin-symmetric convex bodies in $\mathbb{R}^n$ with $|K \cap \theta^{\perp}| \le |L|\theta^{\perp}|$ for all $\theta \in S^{n-1}$, and $K$ has outradius $R$ and $L$ has inradius $r$, then $|K| \le \frac{R}{r}|L|.$ In particular, if $K, L$ are in John's position, this yields the bound $|K| \le \sqrt{n}|L|.$ In this note, we prove that, even without the assumption of being in John's position, we obtain a bound with this same $\sqrt{n}$ order of magnitude.

\begin{theorem}\label{mainthm}
Let $K, L$ be convex bodies in $\mathbb{R}^n$, with $K$ centered, such that \begin{align*}
    |K \cap \theta^{\perp}| &\le |L|\theta^{\perp}|
\end{align*} for all $\theta \in S^{n-1}$. Then, \begin{align*}
    |K| &\le c\sqrt{n}|L|
\end{align*} for an absolute constant $c > 0$.
\end{theorem}

Theorem~\ref{mainthm} is sharp. Indeed, this follows from the construction of Ball \cite{ballshadows} for the isomorphic Shephard problem, since the area of a projection is always at least that of the corresponding section, which means that his example also works for this problem.

\section{Preliminaries and auxiliary results}

\subsection{General Terminology} Here, $S^{n-1}$ denotes the unit sphere in $\mathbb{R}^n$ and $\omega_n$ denotes the volume of the unit ball in $\mathbb{R}^n.$ The ambient dimension satisfies $n\ge 2$ throughout. A set $K$ in $\mathbb{R}^n$ is called convex if the interval joining any two points in $K$ is also contained in $K$. If $K$ is also compact and has non-empty interior, $K$ is  called a convex body. An origin-symmetric convex body $K$ satisfies $x \in K \Leftrightarrow -x \in K$. A convex body $K$ is said to be centered if its barycenter is at the origin.

The Minkowski functional of a convex body $K$ defined as \begin{align*}
\norm{x}_K = \min \{a\ge 0: x \in aK\}
\end{align*} for $x \in \mathbb{R}^n$, and $\rho_K(x) = \norm{x}_K^{-1}$ will be the radial function. If $0 \in \text{int}(K)$, then for $\theta \in S^{n-1}$, $\rho_K(\theta)$ is the distance from the origin to $\partial K$ in the direction of $\theta$. Next, the support function $h_K$ of $K$ is defined by \begin{align*}
h_K(x) = \max_{\xi \in K}\langle x,\xi \rangle.
\end{align*} The polar body of $K$ is defined by \begin{align*}
    K^{\circ} = \{ x \in \mathbb{R}^n: h_K(x) \le 1\}.
\end{align*}The Gauss map of $K$ is the map $\nu_K: \partial K \to S^{n-1}$ that sends $y \in \partial K$ to the set of normal vectors to $K$ at $y$. The surface area measure $S_K$ is then defined by $S_K(E) = H_{n-1}(\nu_K^{-1}(E))$ for all measurable $E \subseteq S^{n-1}$, where $H_{n-1}$ is the $(n-1)-$dimensional Hausdorff measure on $\mathbb{R}^n$. 

Given a convex body $K$ in $\mathbb{R}^n$, the isotropic position of $K$ is defined as the (unique up to orthogonal transformations) affine image $\tilde{K}$ of $K$ with volume 1 and barycenter at the origin such that \begin{align*}
\int_{\tilde{K}}x_ix_j dx = L_K^2 \delta_{ij}
\end{align*} for all $i, j \in \{1,..,n\}$ and some constant $L_K > 0$. Here $x = (x_1, ..., x_n)$ are the coordinates in $\mathbb{R}^n$ and $\delta_{ij}$ is the Kronecker delta symbol. If $K$ is in isotropic position to begin with, we call $K$ isotropic. We will denote by $L_K$ the isotropic constant of $K$. As mentioned earlier, Klartag and Lehec \cite{KlartagLehec2025SlicingGuan} proved \begin{align}\label{slicing}
    L_K &\le C
\end{align} for all convex bodies $K$ in any dimension $n$, where $C > 0$ is an absolute constant. This settled a longstanding question posed by Bourgain \cite{Bourgain1986MaximalConvexBodies}, \cite{Bourgain1987BanachHarmonic}, who also gave the nontrivial estimate $L_K \le c n^{\frac{1}{4}} \log n$ \cite{bourgainbound}. Throughout, $c, C, c_1,$ etc. will denote absolute positive constants.

\subsection{The $L_2$-centroid body}

We first recall the basic facts about the $L_2$-centroid body which will be used in the proof of the main theorem. If $K$ is a centered convex body in $\mathbb{R}^n$, its normalized $L_2$-centroid body $Z_2(K)$ is the origin-symmetric ellipsoid whose support function is given by
\begin{equation*}\label{Z2def}
    h_{Z_2(K)}(\theta)
    =
    \left(\frac{1}{|K|}\int_K |\langle x,\theta\rangle|^2\,dx\right)^{1/2},
    \qquad \theta\in S^{n-1}.
\end{equation*}
Equivalently, if
\begin{equation*}\label{covariancematrixdef}
    M_K
    =
    \frac{1}{|K|}\int_K x\otimes x\,dx
\end{equation*}
is the covariance matrix of the uniform measure on $K$, then
\begin{equation*}\label{Z2covariance}
    h_{Z_2(K)}(\theta)
    =
    \langle M_K\theta,\theta\rangle^{1/2},
    \qquad
    Z_2(K)=M_K^{1/2}B_2^n.
\end{equation*}
%In particular, $Z_2(K)$ is affine-covariant: for every invertible linear map $T$,
%\begin{equation}\label{Z2affinecovariance}
%    Z_2(TK)=T Z_2(K).
%\end{equation}

We shall use the following estimate of Hensley \cite{hensley}. In the normalized case $|K|=1$, it says that $h_{Z_2(K)}(\theta)$ is comparable to the reciprocal of the central hyperplane section in direction $\theta$. After scaling, this gives the following form.

\begin{proposition}[Hensley's estimate \cite{hensley}]\label{HensleyZ2Lemma}
Let $K$ be a centered convex body in $\mathbb{R}^n$. Then, for every $\theta\in S^{n-1}$,
\begin{equation*}\label{HensleyZ2}
    c_1|K|
    \le
    |K\cap \theta^\perp|\, h_{Z_2(K)}(\theta)
    \le
    c_2|K|.
\end{equation*}
\end{proposition}

We remark that Hensley proved this when $K$ is origin-symmetric. The extension to the centered case, along with a discussion of explicit values for $c_1, c_2$, was undertaken by Fradelizi \cite{Fradelizi}.

We also record the relation between $Z_2(K)$ and the isotropic constant. For centered convex bodies $K$ in $\mathbb{R}^n$,
\begin{equation}\label{isotropicconstantdef}
    L_K^n
    =
    \frac{\det(M_K)^{1/2}}{|K|}.
\end{equation}
Equivalently, after rescaling $K$ to have volume one and applying a volume-preserving linear transformation, the covariance matrix becomes $L_K^2 I_n$.

\begin{lemma}\label{Z2VolumeLemma}
For every centered convex body $K$ in $\mathbb R^n$,
\begin{equation*}\label{Z2Volume}
    |Z_2(K)|=\omega_n L_K^n |K|,
\end{equation*}
where $\omega_n=|B_2^n|$.
\end{lemma}

\begin{proof}
Since $Z_2(K)=M_K^{1/2}B_2^n$, we have
\begin{equation*}
    |Z_2(K)|
    =
    \omega_n \det(M_K)^{1/2}.
\end{equation*}
Using \eqref{isotropicconstantdef}, this becomes
\begin{equation*}
    |Z_2(K)|=\omega_n L_K^n |K|,
\end{equation*} as desired.
\end{proof}

\subsection{Projection bodies}

If $L$ is a convex body in $\mathbb{R}^n$, its projection body $\Pi L$ is the origin-symmetric convex body whose support function is given by
\begin{equation*}\label{projectionbodydef}
    h_{\Pi L}(\theta)
    =
    |L|\theta^{\perp}|,
    \qquad \theta\in S^{n-1}.
\end{equation*}
Equivalently, by Cauchy's projection formula,
\begin{equation*}\label{CauchyProjectionFormula}
    h_{\Pi L}(x)
    =
    \frac{1}{2}\int_{S^{n-1}} |\langle x,u\rangle|\,dS_L(u),
    \qquad x\in \mathbb R^n,
\end{equation*}
where $S_L$ denotes the surface area measure of $L$. In particular, $h_{\Pi L}$ is an $L_1$-type norm on $\mathbb R^n$. Hence, the polar projection body
\begin{equation*}
    \Pi^\circ L
    := (\Pi L)^{\circ} = 
    \{x\in \mathbb R^n:h_{\Pi L}(x)\le 1\}
\end{equation*}
is the unit ball of an $n$-dimensional subspace of $L_1$.

We shall use the following theorem of Ball \cite{Ball1991VolumeRatios} and Bourgain and Milman \cite{BourgainMilman1987VolumeRatio}: unit balls of finite-dimensional subspaces of $L_1$ have uniformly bounded volume ratio.

\begin{theorem}[\cite{Ball1991VolumeRatios}, \cite{BourgainMilman1987VolumeRatio}]\label{BallL1VolumeRatioLemma}
Let $Q\subset \mathbb R^n$ be the unit ball of an $n$-dimensional subspace of $L_1$. Then there exists an origin-symmetric ellipsoid $E\subset Q$ such that
\begin{equation*}\label{BallL1VolumeRatio}
    |Q|\le C^n |E|,
\end{equation*}
where $C>0$ is an absolute constant.
\end{theorem}

Bourgain and Milman \cite{BourgainMilman1987VolumeRatio} proved that the volume ratio of
a finite-dimensional normed spaced can be bounded above, solely in terms of the cotype-2
constant of the space. In particular, finite-dimensional
subspaces of $L_1$ have uniformly bounded volume ratios. Ball \cite{Ball1991VolumeRatios} proved the isometric form of this result: for each $p \in [1, \infty]$, $\ell_p^n$ has maximal volume ratio among
$n$-dimensional subspaces of $L_p$

Applying this to $Q=\Pi^\circ L$, we obtain the following consequence.

\begin{corollary}\label{PolarProjectionVolumeRatioCorollary}
For every convex body $L\subset \mathbb R^n$, there exists an origin-symmetric ellipsoid $E\subset \Pi^\circ L$ such that
\begin{equation*}\label{PolarProjectionVolumeRatio}
    |\Pi^\circ L|\le C^n |E|.
\end{equation*}
\end{corollary}

We shall also use Zhang's projection inequality \cite{Zhang1991RestrictedChord} for the main theorem. It gives the sharp lower bound for the volume product of a convex body and the polar of its projection body.

\begin{proposition}[Zhang's projection inequality \cite{Zhang1991RestrictedChord}]\label{ZhangProjectionInequalityLemma}
For every convex body $L\subset \mathbb R^n$,
\begin{equation*}\label{ZhangProjectionInequality}
    |L|^{n-1}|\Pi^\circ L|
    \ge
    \frac{1}{n^n}\binom{2n}{n}.
\end{equation*}
Moreover, equality holds precisely for simplices.
\end{proposition}

\subsection{Ellipsoid Comparison}

We record a simple comparison lemma for ellipsoids.

\begin{lemma}\label{EllipsoidRadialSupportLemma}
Let $E$ and $F$ be origin-symmetric ellipsoids in $\mathbb R^n$. Suppose that
\begin{equation*}
    \rho_E(\theta)\le h_F(\theta)
    \qquad\text{for all } \theta\in S^{n-1}.
\end{equation*}
Then
\begin{equation*}
    |E|\le |F|.
\end{equation*}
\end{lemma}

\begin{proof}
Write
\begin{equation*}
    E=A B_2^n,
    \qquad
    F=B B_2^n,
\end{equation*}
where $A$ and $B$ are symmetric positive definite matrices. Then
\begin{equation*}
    \rho_E(\theta)=\frac{1}{\|A^{-1}\theta\|_2},
    \qquad
    h_F(\theta)=\|B\theta\|_2.
\end{equation*}
Thus the hypothesis gives
\begin{equation*}
    1\le \|A^{-1}\theta\|_2\|B\theta\|_2
    \qquad\text{for all } \theta\in S^{n-1}.
\end{equation*}

Let $a_1\le \cdots \le a_n$ be the eigenvalues of $A$, and let
$b_1\le \cdots \le b_n$ be the eigenvalues of $B$. For each
$k=1,\ldots,n$, let $U_k$ be the span of the eigenvectors of $A$
corresponding to $a_k,\ldots,a_n$, and let $V_k$ be the span of the
eigenvectors of $B$ corresponding to $b_1,\ldots,b_k$. Since
\begin{equation*}
    \dim U_k+\dim V_k=(n-k+1)+k=n+1,
\end{equation*}
there exists a unit vector $\theta_k\in U_k\cap V_k$. For this vector,
\begin{equation*}
    \|A^{-1}\theta_k\|_2\le \frac{1}{a_k},
    \qquad
    \|B\theta_k\|_2\le b_k.
\end{equation*}
Therefore,
\begin{equation*}
    1\le \|A^{-1}\theta_k\|_2\|B\theta_k\|_2
    \le \frac{b_k}{a_k}.
\end{equation*}
Hence $a_k\le b_k$ for every $k$. Multiplying over $k$ gives
\begin{equation*}
    \det A\le \det B.
\end{equation*}
Since $|E|=\omega_n\det A$ and $|F|=\omega_n\det B$, this proves
$|E|\le |F|$.
\end{proof}

\section{Proof of the main theorem}

\begin{proof}[Proof of Theorem~\ref{mainthm}]
We prove the estimate
\begin{equation}\label{LKrefinedbound}
    |K|\le C\sqrt n\, L_K^{n/(n-1)} |L|.
\end{equation}
The theorem then follows from the slicing theorem, since $L_K\le C$~\eqref{slicing}.

Set
\begin{equation*}
    Q=\Pi^\circ L.
\end{equation*}
By the definition of the polar projection body,
\begin{equation*}
    \rho_Q(\theta)
    =
    \frac{1}{|L|\theta^{\perp}|}
    \qquad\text{for all } \theta\in S^{n-1}.
\end{equation*}
Recalling the hypothesis
\begin{equation*}
    |K\cap \theta^\perp|\le |L|\theta^{\perp}|
    \qquad\text{for all } \theta\in S^{n-1},
\end{equation*}
it follows that
\begin{equation*}
    \rho_Q(\theta)
    =
    \frac{1}{|L|\theta^{\perp}|}
    \le
    \frac{1}{|K\cap \theta^\perp|}.
\end{equation*}
By Hensley's estimate, Proposition~\ref{HensleyZ2Lemma},
\begin{equation*}
    \frac{1}{|K\cap \theta^\perp|}
    \le
    C\frac{h_{Z_2(K)}(\theta)}{|K|}.
\end{equation*}
Consequently,
\begin{equation}\label{rhoQZ2comparison}
    \rho_Q(\theta)
    \le
    C\frac{h_{Z_2(K)}(\theta)}{|K|}
    \qquad\text{for all } \theta\in S^{n-1}.
\end{equation}

Define the ellipsoid
\begin{equation*}
    F=\frac{C}{|K|}Z_2(K),
\end{equation*}
with the same constant $C$ as in
\eqref{rhoQZ2comparison}. Then,
\begin{equation}\label{rhoQhFcomparison}
    \rho_Q(\theta)\le h_F(\theta)
    \qquad\text{for all } \theta\in S^{n-1}.
\end{equation}

By the volume ratio theorem for subspaces of $L_1$, Theorem~\ref{BallL1VolumeRatioLemma}, applied to
$Q=\Pi^\circ L$, there exists an origin-symmetric ellipsoid $J\subset Q$
such that
\begin{equation}\label{BallApplied}
    |Q|\le C^n |J|.
\end{equation}
Since $J\subset Q$, we have
\begin{equation*}
    \rho_J(\theta)\le \rho_Q(\theta)
    \qquad\text{for all } \theta\in S^{n-1},
\end{equation*}
which, combined with \eqref{rhoQhFcomparison}, yields
\begin{equation*}
    \rho_J(\theta)\le h_F(\theta)
    \qquad\text{for all } \theta\in S^{n-1}.
\end{equation*}
By Lemma~\ref{EllipsoidRadialSupportLemma},
\begin{equation*}
    |J|\le |F|.
\end{equation*}
Therefore, by \eqref{BallApplied},
\begin{equation*}\label{PiPolarUpperByF}
    |\Pi^\circ L|=|Q|\le C^n |F|.
\end{equation*}

We now compute $|F|$. Since
\begin{equation*}
    F=\frac{C}{|K|}Z_2(K),
\end{equation*}
Lemma~\ref{Z2VolumeLemma} gives
\begin{equation*}
    |F|
    =
    \left(\frac{C}{|K|}\right)^n |Z_2(K)|
    =
    C^n\omega_n L_K^n |K|^{1-n}.
\end{equation*}
Thus
\begin{equation}\label{PiPolarUpper}
    |\Pi^\circ L|
    \le
    C^n\omega_n L_K^n |K|^{1-n}.
\end{equation}

On the other hand, Zhang's projection inequality, Proposition~\ref{ZhangProjectionInequalityLemma}, gives
\begin{equation}\label{PiPolarLower}
    |\Pi^\circ L|
    \ge
    \frac{1}{n^n}\binom{2n}{n}|L|^{1-n}.
\end{equation}
Combining \eqref{PiPolarUpper} and \eqref{PiPolarLower}, we obtain
\begin{equation*}
    \frac{1}{n^n}\binom{2n}{n}|L|^{1-n}
    \le
    C^n\omega_n L_K^n |K|^{1-n}.
\end{equation*}
Equivalently,
\begin{equation}\label{volumecomparisonbeforestirling}
    \left(\frac{|K|}{|L|}\right)^{n-1}
    \le
    C^n
    \frac{n^n\omega_n}{\binom{2n}{n}}
    L_K^n.
\end{equation}

It remains only to estimate the dimensional factor. We use the standard bounds
\begin{equation*}
    \omega_n^{1/n}\le \frac{C}{\sqrt n}
\end{equation*}
and
\begin{equation*}
    \binom{2n}{n}\ge \frac{c4^n}{\sqrt n}.
\end{equation*}
These imply
\begin{equation*}
    \left(
        \frac{n^n\omega_n}{\binom{2n}{n}}
    \right)^{1/(n-1)}
    \le C\sqrt n.
\end{equation*}
Taking $(n-1)$-st roots in \eqref{volumecomparisonbeforestirling} therefore yields
\begin{equation*}
    \frac{|K|}{|L|}
    \le
    C\sqrt n\,L_K^{n/(n-1)}.
\end{equation*}
This proves \eqref{LKrefinedbound} and the theorem.
\end{proof}

\section{Acknowledgements}

The author was supported in part by Simons Foundation Collaboration Grant 601948 DJ. The author used ChatGPT Pro 5.5 during the development of this work, including to suggest proof strategies and to find references.

\bibliographystyle{alpha}
\bibliography{biblio.bib}

@article {busemannpetty,
    AUTHOR = {Busemann, H. and Petty, C. M.},
     TITLE = {Problems on convex bodies},
   JOURNAL = {Math. Scand.},
  FJOURNAL = {Mathematica Scandinavica},
    VOLUME = {4},
      YEAR = {1956},
     PAGES = {88--94},
      ISSN = {0025-5521},
   MRCLASS = {52.0X},
  MRNUMBER = {0084791},
MRREVIEWER = {W. Fenchel},
       DOI = {10.7146/math.scand.a-10457},
       URL = {https://doi.org/10.7146/math.scand.a-10457},
}

@article{Fradelizi,
  author  = {Fradelizi, Matthieu},
  title   = {Hyperplane Sections of Convex Bodies in Isotropic Position},
  journal = {Beitr{\"a}ge zur Algebra und Geometrie},
  volume  = {40},
  number  = {1},
  pages   = {163--183},
  year    = {1999}
}

@article{Zhang1991RestrictedChord,
  author  = {Zhang, Gaoyong},
  title   = {Restricted Chord Projection and Affine Inequalities},
  journal = {Geometriae Dedicata},
  volume  = {39},
  number  = {2},
  pages   = {213--222},
  year    = {1991},
  doi     = {10.1007/BF00182294}
}

@article{BourgainMilman1987VolumeRatio,
  author  = {Bourgain, J. and Milman, V. D.},
  title   = {New Volume Ratio Properties for Convex Symmetric Bodies in $\mathbb{R}^n$},
  journal = {Inventiones Mathematicae},
  volume  = {88},
  number  = {2},
  pages   = {319--340},
  year    = {1987},
  doi     = {10.1007/BF01388911}
}

@article{Ball1991VolumeRatios,
  author  = {Ball, Keith},
  title   = {Volume Ratios and a Reverse Isoperimetric Inequality},
  journal = {Journal of the London Mathematical Society},
  series  = {2},
  volume  = {44},
  number  = {2},
  pages   = {351--359},
  year    = {1991},
  doi     = {10.1112/jlms/s2-44.2.351}
}

@article{hensley,
  author  = {Hensley, Douglas},
  title   = {Slicing Convex Bodies---Bounds for Slice Area in Terms of the Body's Covariance},
  journal = {Proceedings of the American Mathematical Society},
  volume  = {79},
  number  = {4},
  pages   = {619--625},
  year    = {1980}
}

@article{shepard,
    AUTHOR = {Shephard, G. C.},
     TITLE = {Shadow systems of convex sets},
   JOURNAL = {Israel J. Math.},
  FJOURNAL = {Israel Journal of Mathematics},
    VOLUME = {2},
      YEAR = {1964},
     PAGES = {229--236},
      ISSN = {0021-2172},
   MRCLASS = {52.30},
  MRNUMBER = {0179686},
MRREVIEWER = {A. M. Macbeath},
       DOI = {10.1007/BF02759738},
       URL = {https://doi.org/10.1007/BF02759738},
}

@article {papa,
    AUTHOR = {Papadimitrakis, Michael},
     TITLE = {On the {B}usemann-{P}etty problem about convex, centrally
              symmetric bodies in {$\bold R^n$}},
   JOURNAL = {Mathematika},
  FJOURNAL = {Mathematika. A Journal of Pure and Applied Mathematics},
    VOLUME = {39},
      YEAR = {1992},
    NUMBER = {2},
     PAGES = {258--266},
      ISSN = {0025-5793},
   MRCLASS = {52A38 (52A40)},
  MRNUMBER = {1203282},
MRREVIEWER = {R. J. Gardner},
       DOI = {10.1112/S0025579300014996},
       URL = {https://doi.org/10.1112/S0025579300014996},
}

@article{giankold,
    AUTHOR = {Giannopoulos, Apostolos and Koldobsky, Alexander},
     TITLE = {Variants of the {B}usemann-{P}etty problem and of the
              {S}hephard problem},
   JOURNAL = {Int. Math. Res. Not. IMRN},
  FJOURNAL = {International Mathematics Research Notices. IMRN},
      YEAR = {2017},
      VOLUME = {2017},
    NUMBER = {3},
     PAGES = {921--943},
      ISSN = {1073-7928},
   MRCLASS = {52A40},
  MRNUMBER = {3658155},
MRREVIEWER = {Stefano Campi},
       DOI = {10.1093/imrn/rnw046},
       URL = {https://doi.org/10.1093/imrn/rnw046},
}

@article {Gardner1,
    AUTHOR = {Gardner, R. J.},
     TITLE = {Intersection bodies and the {B}usemann-{P}etty problem},
   JOURNAL = {Trans. Amer. Math. Soc.},
  FJOURNAL = {Transactions of the American Mathematical Society},
    VOLUME = {342},
      YEAR = {1994},
    NUMBER = {1},
     PAGES = {435--445},
      ISSN = {0002-9947},
   MRCLASS = {52A38 (52A40)},
  MRNUMBER = {1201126},
MRREVIEWER = {W. J. Firey},
       DOI = {10.2307/2154703},
       URL = {https://doi.org/10.2307/2154703},
}

@article {Zhang1,
    AUTHOR = {Zhang, Gaoyong},
     TITLE = {Intersection bodies and the {B}usemann-{P}etty inequalities in
              {${\bf R}^4$}},
   JOURNAL = {Ann. of Math. (2)},
  FJOURNAL = {Annals of Mathematics. Second Series},
    VOLUME = {140},
      YEAR = {1994},
    NUMBER = {2},
     PAGES = {331--346},
      ISSN = {0003-486X},
   MRCLASS = {52A38 (52A20 52A40 52B12)},
  MRNUMBER = {1298716},
MRREVIEWER = {P. R. Goodey},
       DOI = {10.2307/2118603},
       URL = {https://doi.org/10.2307/2118603},
}

@article {Gardner2,
    AUTHOR = {Gardner, R. J.},
     TITLE = {A positive answer to the {B}usemann-{P}etty problem in three
              dimensions},
   JOURNAL = {Ann. of Math. (2)},
  FJOURNAL = {Annals of Mathematics. Second Series},
    VOLUME = {140},
      YEAR = {1994},
    NUMBER = {2},
     PAGES = {435--447},
      ISSN = {0003-486X},
   MRCLASS = {52A38 (52A15 52A40 52B12)},
  MRNUMBER = {1298719},
MRREVIEWER = {P. R. Goodey},
       DOI = {10.2307/2118606},
       URL = {https://doi.org/10.2307/2118606},
}

@article {Zhang2,
    AUTHOR = {Zhang, Gaoyong},
     TITLE = {A positive solution to the {B}usemann-{P}etty problem in
              {$\bf{R}^4$}},
   JOURNAL = {Ann. of Math. (2)},
  FJOURNAL = {Annals of Mathematics. Second Series},
    VOLUME = {149},
      YEAR = {1999},
    NUMBER = {2},
     PAGES = {535--543},
      ISSN = {0003-486X},
   MRCLASS = {52A38 (44A12 52A20)},
  MRNUMBER = {1689339},
MRREVIEWER = {Apostolos A. Giannopoulos},
       DOI = {10.2307/120974},
       URL = {https://doi.org/10.2307/120974},
}

@article {schneider,
    AUTHOR = {Schneider, Rolf},
     TITLE = {Zur einem {P}roblem von {S}hephard \"uber die {P}rojektionen
              konvexer {K}\"orper},
   JOURNAL = {Math. Z.},
  FJOURNAL = {Mathematische Zeitschrift},
    VOLUME = {101},
      YEAR = {1967},
     PAGES = {71--82},
      ISSN = {0025-5874},
   MRCLASS = {52.40},
  MRNUMBER = {0218976},
MRREVIEWER = {K. Strubecker},
       DOI = {10.1007/BF01135693},
       URL = {https://doi.org/10.1007/BF01135693},
}

@article{Hosle,
  author  = {Hosle, Johannes},
  title   = {On the Comparison of Measures of Convex Bodies via Projections and Sections},
  journal = {International Mathematics Research Notices},
  volume  = {2021},
  number  = {17},
  pages   = {13046--13074},
  year    = {2021},
  doi     = {10.1093/imrn/rnz215}
}

@article {ballshadows,
    AUTHOR = {Ball, Keith},
     TITLE = {Shadows of convex bodies},
   JOURNAL = {Trans. Amer. Math. Soc.},
  FJOURNAL = {Transactions of the American Mathematical Society},
    VOLUME = {327},
      YEAR = {1991},
    NUMBER = {2},
     PAGES = {891--901},
      ISSN = {0002-9947},
   MRCLASS = {52A40 (52A20)},
  MRNUMBER = {1035998},
MRREVIEWER = {B\'ela Uhrin},
       DOI = {10.2307/2001829},
       URL = {https://doi.org/10.2307/2001829},
}

@incollection {milmanpajor,
    AUTHOR = {Milman, V. D. and Pajor, A.},
     TITLE = {Isotropic position and inertia ellipsoids and zonoids of the
              unit ball of a normed {$n$}-dimensional space},
 BOOKTITLE = {Geometric aspects of functional analysis (1987--88)},
    SERIES = {Lecture Notes in Math.},
    VOLUME = {1376},
     PAGES = {64--104},
 PUBLISHER = {Springer, Berlin},
      YEAR = {1989},
   MRCLASS = {52A20 (46B20)},
  MRNUMBER = {1008717},
MRREVIEWER = {W. J. Firey},
       DOI = {10.1007/BFb0090049},
       URL = {https://doi.org/10.1007/BFb0090049},
}

@incollection {petty,
    AUTHOR = {Petty, C. M.},
     TITLE = {Projection bodies},
 BOOKTITLE = {Proc. {C}olloquium on {C}onvexity ({C}openhagen, 1965)},
     PAGES = {234--241},
 PUBLISHER = {Kobenhavns Univ. Mat. Inst., Copenhagen},
      YEAR = {1967},
   MRCLASS = {52.40},
  MRNUMBER = {0216369},
MRREVIEWER = {G. D. Chakerian},
}

@article {GKS,
    AUTHOR = {Gardner, R. J. and Koldobsky, A. and Schlumprecht, T.},
     TITLE = {An analytic solution to the {B}usemann-{P}etty problem on
              sections of convex bodies},
   JOURNAL = {Ann. of Math. (2)},
  FJOURNAL = {Annals of Mathematics. Second Series},
    VOLUME = {149},
      YEAR = {1999},
    NUMBER = {2},
     PAGES = {691--703},
      ISSN = {0003-486X},
   MRCLASS = {52A38 (44A12 52A20)},
  MRNUMBER = {1689343},
MRREVIEWER = {Apostolos A. Giannopoulos},
       DOI = {10.2307/120978},
       URL = {https://doi.org/10.2307/120978},
}

@article{KlartagLehec2025SlicingGuan,
  author  = {Klartag, Boaz and Lehec, Joseph},
  title   = {Affirmative Resolution of {Bourgain}'s Slicing Problem Using {Guan}'s Bound},
  journal = {Geometric and Functional Analysis},
  volume  = {35},
  number  = {4},
  pages   = {1147--1168},
  year    = {2025},
  doi     = {10.1007/s00039-025-00718-w}
}

@incollection {bourgainbound,
    AUTHOR = {Bourgain, J.},
     TITLE = {On the distribution of polynomials on high-dimensional convex
              sets},
 BOOKTITLE = {Geometric aspects of functional analysis (1989--90)},
    SERIES = {Lecture Notes in Math.},
    VOLUME = {1469},
     PAGES = {127--137},
 PUBLISHER = {Springer, Berlin},
      YEAR = {1991},
   MRCLASS = {52A21 (46B07 46E30 52A20)},
  MRNUMBER = {1122617},
MRREVIEWER = {Wolfgang Lusky},
       DOI = {10.1007/BFb0089219},
       URL = {https://doi.org/10.1007/BFb0089219},
}

@article{Bourgain1986MaximalConvexBodies,
  author  = {Bourgain, J.},
  title   = {On High-Dimensional Maximal Functions Associated to Convex Bodies},
  journal = {American Journal of Mathematics},
  volume  = {108},
  number  = {6},
  pages   = {1467--1476},
  year    = {1986},
  doi     = {10.2307/2374532}
}

@inproceedings{Bourgain1987BanachHarmonic,
  author    = {Bourgain, J.},
  title     = {Geometry of Banach Spaces and Harmonic Analysis},
  booktitle = {Proceedings of the International Congress of Mathematicians, Berkeley 1986},
  pages     = {871--878},
  publisher = {American Mathematical Society},
  address   = {Providence, RI},
  year      = {1987}
}
\end{document}